\documentclass[11pt]{amsart}
\usepackage{hyperref}
\usepackage{amsfonts}
\usepackage{amsmath}
\usepackage{amssymb}
\usepackage{amscd}
\usepackage{graphicx}
\usepackage{enumitem}
\usepackage{latexsym}
\usepackage{color}
\usepackage{epsfig}
\usepackage{amsopn}
\usepackage{xypic}
\usepackage{mathrsfs}
\usepackage{array}

\usepackage{booktabs}
\usepackage{longtable}
\theoremstyle{plain}
\newtheorem{lemma}{Lemma}[section]
\newtheorem*{theorem*}{Theorem}
\newtheorem*{lemma*}{Lemma}
\newtheorem*{proposition*}{Proposition}
\newtheorem*{conjecture*}{Conjecture}
\newtheorem*{corollary*}{Corollary}
\newtheorem*{problem*}{Problem}
\newtheorem{theorem}[lemma]{Theorem}
\newtheorem{conjecture}[lemma]{Conjecture}
\newtheorem{corollary}[lemma]{Corollary}
\newtheorem{proposition}[lemma]{Proposition}

\theoremstyle{definition}
\newtheorem{definition}[lemma]{Definition}
\newtheorem{example}[lemma]{Example}
\newtheorem{remark}[lemma]{Remark}

\newtheorem{notation}{Notation}

\newcommand{\F}[1]{\mathscr{#1}}

\newcommand{\Z}{\mathbb{Z}}

\renewcommand{\F}{\mathbb{F}}
\newcommand{\C}{\mathbb{C}}
\newcommand{\Q}{\mathbb{Q}}

\newcommand{\OO}{\mathcal{O}}
\newcommand{\te}{\otimes}

\newcommand{\cF}{\mathcal F}

\newcommand{\cM}{\mathcal M}

\newcommand{\cP}{\mathcal P}

\newcommand{\cE}{\mathcal{E}}

\newcommand{\ZZ}{\mathbb{Z}}

\renewcommand{\P}{\mathbb{P}}

\newcommand{\PP}{\mathbb{P}}

\DeclareMathOperator{\Bl}{Bl}
\DeclareMathOperator{\ch}{ch}

\DeclareMathOperator{\Hom}{Hom}

\DeclareMathOperator{\Pic}{Pic}

\DeclareMathOperator{\NE}{NE}

\DeclareMathOperator{\rk}{rk}

\DeclareMathOperator{\Ext}{Ext}

\DeclareMathOperator{\sHom}{\mathcal{H} \textit{om}}

\DeclareMathOperator{\Nef}{Nef}

\DeclareMathOperator{\num}{num}
\DeclareMathOperator{\td}{td}

\begin{document}

\date{\today}
\author[I. Coskun]{Izzet Coskun}
\address{Department of Mathematics, Statistics and CS \\University of Illinois at Chicago, Chicago, IL 60607}
\email{coskun@math.uic.edu}
\author[J. Huizenga]{Jack Huizenga}
\address{Department of Mathematics, The Pennsylvania State University, University Park, PA 16802}
\email{huizenga@psu.edu}
\thanks{During the preparation of this article the first author was partially supported by the NSF CAREER grant DMS-0950951535 and NSF grant DMS-1500031, and the second author was partially supported by a National Science Foundation Mathematical Sciences Postdoctoral Research Fellowship DMS-1204066 and an NSA\ Young Investigator Grant H98230-16-1-0306}

\title{Weak Brill-Noether for rational surfaces}
\dedicatory{To Lawrence Ein, with gratitude and admiration, on the occasion of his sixtieth birthday}

\begin{abstract}
A moduli space of sheaves satisfies weak Brill-Noether if the general sheaf in the moduli space has no cohomology.  G\"{o}ttsche and Hirschowitz prove that on $\PP^2$ every moduli space of Gieseker semistable sheaves of rank at least two and Euler characteristic zero satisfies weak Brill-Noether. In this paper, we give sufficient conditions for weak Brill-Noether to hold on rational surfaces. We completely characterize Chern characters on Hirzebruch surfaces for which weak Brill-Noether holds. We also prove that on a del Pezzo surface of degree at least 4 weak Brill-Noether holds if the first Chern class is nef.  
\end{abstract}

\maketitle

\setcounter{tocdepth}{1}
\tableofcontents

\section{Introduction}

Let $X$ be a smooth, complex projective surface and let $H$ be an ample divisor on $X$. Let ${\bf v}\in K_{\num}(X)$ be the (numerical) Chern character of a Gieseker semistable sheaf on $X$ such that the Euler characteristic satisfies $\chi({\bf v})=0$.  Let $M_{X,H}({\bf v})$ denote the moduli space of Gieseker semistable sheaves on $X$ with Chern character ${\bf v}$.  We will call ${\bf v}$ \emph{stable} if there is a semistable sheaf of character ${\bf v}$.

\begin{definition}
The moduli space $M_{X, H}({\bf v})$ satisfies {\em weak Brill-Noether} if there exists a sheaf $E \in  M_{X,H}({\bf v})$ such that $H^i(X, E)=0$ for all $i$. 
\end{definition}

By  semicontinuity, if $E$ is any sheaf with no cohomology, then the cohomology also vanishes for the general sheaf in any component of  $M_{X,H}({\bf v})$ that contains $E$.  Since the cohomology vanishes identically, for weak Brill-Noether to hold, we must have  $\chi({\bf v})=0$.

The weak Brill-Noether  property is the key ingredient in constructing effective theta divisors on the moduli spaces $M_{X,H}({\bf v})$ and  plays a central role in describing effective cones of moduli spaces and strange duality. Assume that $M_{X, H}({\bf v})$ is irreducible. The locus
$$\Theta = \{ E \in M_{X,H}({\bf v}) | h^1(X, E) \not= 0\}$$ is called the {\em theta locus} and is an effective divisor when weak Brill-Noether holds for $M_{X,H}({\bf v})$.  In this paper, we study when Chern characters ${\bf v}$ satisfy weak Brill-Noether on rational surfaces $X$. We  completely classify Chern characters ${\bf v}$ satisfying weak Brill-Noether on Hirzebruch surfaces.   For general rational surfaces, we give sufficient conditions on Chern characters that guarantee that weak Brill-Noether holds. In particular, we show that if $X$ is a del Pezzo surface of degree at least 4 and $\ch_1({\bf v})$ is nef, then weak Brill-Noether holds. We now summarize our results in greater detail and give some examples.

\subsection*{Weak Brill-Noether for rank one sheaves} When the sheaves have rank one, the story is particularly simple.

\begin{proposition}
Let $L$ be a line bundle on a smooth projective surface $X$ such that  $n = \chi(L)\geq 0$.  Let $Z\subset X$ be a general zero-dimensional scheme of length $n$, and set ${\bf v} = \ch(L\te I_Z)$, so that $\chi({\bf v})=0$.  Then $H^i(X, L \otimes I_Z)=0$ for all $i$ if and only if $H^1(X, L)= H^2(X, L)=0$. In particular,  weak Brill-Noether holds for $M_{X,H}({\bf v})$ if and only if there exists a line bundle $M$ with $\ch_1(M)= \ch_1({\bf v})$ and $H^1(X, M) = H^2(X, M) = 0$.
\end{proposition}
\begin{proof}
Any torsion-free rank one sheaf  on a surface is isomorphic to $L \otimes I_Z$ for a line bundle $L$ and an ideal sheaf $I_Z$ of a zero-dimensional scheme. First, suppose  $H^1(X, L)=H^2(X, L)=0$ and $h^0(X, L)=n$.  Let $X^{[n]}$ denote the Hilbert scheme of $n$ points on $X$ and let $Z\in X^{[n]}$ be a general subscheme of length $n$. Consider the restriction sequence $$0\to L\te I_Z\to L \to L|_Z\to 0.$$  Since $Z$ is general, the map $H^0(X, L)\to H^0(X, L|_Z)$ is an isomorphism, and we find $L\te I_Z$ has no cohomology.

Conversely, suppose $H^1(X, L)$ or $H^2(X, L)$ is nonzero.  If $h^0(X, L)>n$, then $h^0(X, L\te I_Z)>0$ for every $Z\in X^{[n]}$.  Thus we may assume $h^0(X, L)\leq n$.  Then since at least one of $H^1(X, L)$ or $H^2(X, L)$ is nonzero and $\chi(L)=n$, we find $H^2(X, L)\neq 0$.  But $H^2(X, L\te I_Z) \cong H^2(X, L)$ for any $Z\in X^{[n]}$. In particular, weak Brill-Noether holds for $M_{X, H} ({\bf v})$ if and only if there exists a line bundle $M$ with $\ch_1(M) = \ch_1({\bf v})$ and $H^1(X, M) = H^2(X, M) =0$. 
\end{proof}

\begin{example}
On $\PP^2$, weak Brill-Noether fails for moduli spaces of rank one sheaves when the slope is $-3$ or less. This failure is intimately tied to the fact that the general rank one sheaf is not locally free. In contrast, G\"{o}ttsche and Hirschowitz \cite{GottscheHirschowitz} prove that if ${\bf v}$  has $\rk({\bf v}) \geq 2$ and $\chi({\bf v})$ is arbitrary, then the general sheaf $E \in M_{\PP^2, \OO(1)}({\bf v})$ has at most one nonzero cohomology group. The signs of the Euler characteristic and the slope determine which cohomology group is nonzero. In particular, if $\chi({\bf v}) =0$, then for the general stable sheaf, all cohomology groups vanish and weak Brill-Noether holds. The purpose of this paper is to generalize this theorem to other rational surfaces.
\end{example}

Our first main result classifies Chern characters on Hirzebruch surfaces that satisfy weak Brill-Noether. Let $H$ be any ample class on the Hirzebruch surface $\F_e =\PP(\OO_{\PP^1} \oplus \OO_{\PP^1}(e))$, $e\geq 0$. Let ${\bf v}$ be a stable Chern character of rank at least $2$ on  $\F_e$ such that $\chi({\bf v})=0$.   Let $$\nu({\bf v}) = \frac{c_1({\bf v})}{r({\bf v})} = \frac{k}{r} E + \frac{l}{r}F$$ denote the total slope of ${\bf v}$, where $E$ is the section of self-intersection $-e$ and $F$ is a fiber. By Walter's Theorem \cite[Theorem 1]{Walter}, the moduli spaces $M_{\F_e, H}({\bf v})$ are irreducible and the general sheaf is locally free. By Serre duality,  we may assume that $$\frac{k}{r}\geq -1 \ \ \mbox{ and if} \  \ \frac{k}{r}=-1, \ \mbox{then} \ \frac{l}{r} \geq -1-\frac{e}{2}.$$   

\begin{theorem}\label{thmIntroHirzebruch}
Let ${\bf v}$ be a stable Chern character that satisfies these inequalities. Then $M_{\F_e, H}({\bf v})$ satisfies weak Brill-Noether if and only if $$\frac{l-ke}{r} = \nu({\bf v}) \cdot E \geq -1.$$
\end{theorem}

As Theorem \ref{thmIntroHirzebruch} demonstrates, the existence of an effective curve $C$ on $X$ such that 
$${\bf w} = \ch(\OO_X(C)), \quad \chi({\bf w}, {\bf v}) > 0, \quad  \mbox{and} \quad \nu({\bf v}) \cdot H >(K_X + C) \cdot H$$
provides an obstruction to weak Brill-Noether for $M_{X,H}({\bf v})$. By Serre duality and stability, $$\Ext^2(\OO_X (C), \cE) \cong \Hom (\cE,  \OO_X(K_X+C))^* = 0.$$ Since $\chi(\OO(C), \cE) >0$, we have $\Hom(\OO_X(C), \cE)\not= 0$.  Composing with the natural map $$\OO_X \longrightarrow \OO_X(C), $$ we see that $\Hom(\OO_X, \cE) = H^0(X, \cE) \not= 0$ for every $\cE \in M_{X,H}({\bf v})$. We remark that if  $M_{X,H}({\bf v})$ is nonempty, $C$ must also satisfy $\nu({\bf v}) \cdot H \geq C \cdot H$. 

In order to prove weak Brill-Noether theorems, we need to ensure that these obstructions vanish. Let $\cP_{X, F}({\bf v})$ denote the stack of  $F$-prioritary sheaves on $X$. For blowups of $\PP^2$, our sharpest result is the following.

\begin{theorem}\label{thmIntroGeneral}
Let $X$ be a blowup of $\PP^2$ at $k$ distinct points $p_1, \dots, p_k$. Let $L$ be the pullback of the hyperplane class of $\PP^2$ and let $E_i$ be the exceptional divisor over $p_i$. Let ${\bf v}\in K(X)$ with $r=r({\bf v})>0$, and write $$\nu({\bf v}) := \frac{c_1({\bf v})}{r({\bf v})} = \delta L - \alpha_1E_1-\cdots -\alpha_k E_k,$$ so that the coefficients $\delta,\alpha_i\in \Q$.  Assume that $\delta\geq 0$ and $\alpha_i \geq 0$ for all $i$.  Suppose that the line bundle $$\lfloor \delta \rfloor L - \lceil \alpha_1\rceil E_1-\cdots - \lceil\alpha_k\rceil E_k$$ has no higher cohomology.  If $\chi({\bf v}) = 0$, then the stack $\cP_{X, L-E_1}({\bf v})$ is nonempty and the general $\cE\in \cP_{X, L-E_1}({\bf v})$ has no cohomology.
\end{theorem}

In particular, when $H$ is an ample divisor on $X$ that satisfies $H \cdot (K_X +L-E_1) < 0$ and ${\bf v}$ is an $H$-stable Chern character satisfying the assumptions of Theorem \ref{thmIntroGeneral}, then $M_{X, H}({\bf v})$ satisfies weak Brill-Noether. On del Pezzo surfaces of large degree we obtain sharper results.

\begin{theorem}\label{thmIntrodelPezzo}
Let $X$ be a del Pezzo surface of degree at least 4.  Let ${\bf v}\in K(X)$ with $\chi({\bf v}) = 0$, and suppose $c_1({\bf v})$ is nef.  Then the stack $\cP_{X, L-E_1}({\bf v})$  is nonempty and a general $\cE\in \cP_{X, L-E_1}({\bf v})$ has no cohomology.
\end{theorem}

The following conjecture may be thought of as a higher rank analogue of the celebrated Segre-Harbourne-Gimigliano-Hirschowitz conjecture \cite{Segre, Harbourne, gimi2, Hirschowitz}.

\begin{conjecture}
Assume $X$ is a general blowup of $\PP^2$ and $c_1({\bf v})$ is nef. Let $F= L -E_1$.  If $H$ is an ample class such that $H \cdot (K_X + F) < 0$, then $M_{X, H}({\bf v})$ satisfies weak Brill-Noether. 
\end{conjecture}

We use two techniques to prove weak Brill-Noether Theorems. First, we give a resolution of the general sheaf on $M_{X, H}({\bf v})$ in terms of a strong exceptional collection satisfying certain cohomology vanishing properties. This method allows us to prove Theorem \ref{thmIntroHirzebruch} and show weak Brill-Noether on arbitrary rational surfaces provided ${\bf v}$ satisfies certain inequalities. The advantage of this method is that it gives a convenient resolution of the general sheaf in $M_{X, H}({\bf v})$. As a consequence, it shows that the moduli space is unirational. The disadvantage is that this method is only applicable when the surface $X$ admits a strong exceptional collection of the desired form. 

Second, we construct an explicit prioritary sheaf with vanishing cohomology as a sum of line bundles.  Walter \cite[Proposition 2]{Walter} proves that  on a birationally ruled surface the stack parameterizing sheaves prioritary with respect to the fiber class  is smooth and irreducible. Assuming that the stable sheaves in $M_{X, H}({\bf v})$ are prioritary, to prove weak Brill-Noether, it suffices to exhibit one prioritary sheaf with vanishing cohomology. Constructing prioritary sheaves is much easier than constructing stable sheaves. In particular, under suitable assumptions, one may construct prioritary sheaves as sums of line bundles. The problem then reduces to the combinatorial problem of finding a prioritary combination of line bundles with no higher cohomology that has the same rank and first Chern class as ${\bf v}$. We solve this problem explicitly for del Pezzo surfaces of degree at least 4. Both of these techniques are applicable much more generally. However, to minimize the  combinatorial complexity, we make additional assumptions on the Chern character ${\bf v}$ when convenient.    

\subsection*{The organization of the paper} In \S \ref{secPrelim}, we collect basic facts about moduli spaces of sheaves and the cohomology of line bundles on rational surfaces. In \S \ref{secResolutions}, we introduce our first method for proving weak Brill-Noether and characterize Chern characters on Hirzebruch surfaces that satisfy weak Brill-Noether. In \S \ref{secPrioritary} and \S \ref{secDelPezzo}, we introduce our second method and show that Chern characters with nef $\ch_1$ satisfy weak Brill-Noether on del Pezzo surfaces of degree at least 4. 

\subsection*{Acknowledgements} We would like to express our gratitude to Lawrence Ein whose unfailing support has been invaluable in our careers.  We would also like to thank Daniel Levine for making comments and corrections on an earlier version of this work.

\section{Preliminaries}\label{secPrelim}
In this section, we recall standard facts concerning Hirzebruch and del Pezzo surfaces and cohomology of line bundles on rational surfaces. We refer the reader to \cite{Beauville}, \cite{CoskunScroll}, \cite{CoskundelPezzo} or \cite{Hartshorne} for more detailed expositions. 

\subsection*{Hirzebruch surfaces}
Let $e \geq 0$ be a nonnegative integer. Let $\F_e$ denote the Hirzebruch surface $\PP(\OO_{\PP^1} \oplus \OO_{\PP^1}(e))$. When $e \geq 1$, let $E$ be the class of the unique section of self intersection $E^2 = -e$ and let $F$ denote the class of a fiber of the projection to $\PP^1$. The surface $\F_0$ is isomorphic to $\PP^1 \times \PP^1$. In that case, let $E$ and $F$ denote the classes of the two rulings. Then $$\Pic(\F_e) \cong \ZZ E \oplus \ZZ F  \quad \mbox{with} \quad E^2 = -e, \quad E \cdot F = 1, \quad F^2 =0.$$ By adjunction,  $$K_{\F_e} = -2 E - (e+2) F.$$ Consequently,  the Riemann-Roch Theorem implies that
$$\chi(\OO_{\F_e}(aE+bF)) = (a+1)(b+1) - e \frac{a (a+1)}{2}.$$
The effective cone of $\F_e$ is generated by $E$ and $F$, consequently $$H^0(\F_e, \OO_{\F_e}(aE+bF)) \not= 0 \quad \mbox{if and only if} \quad a, b \geq 0.$$ By Serre duality, $$H^2(\F_e, \OO_{\F_e}(aE + bF)) \not= 0 \quad \mbox{if and only if} \quad a \leq -2  \ \mbox{and} \ b \leq -2-e.$$  Hence, to compute the cohomology of all line bundles, it suffices to assume that  $a \geq -1$. In this case, since $H^2$ vanishes and we have computed the Euler characteristic, specifying $h^0$ determines the dimension of all cohomology groups. The following theorem (see \cite{CoskunScroll}, \cite[\S V.2]{Hartshorne}) summarizes the answer. 

\begin{theorem}\label{thmCohomologyH}
Let $\OO_{\F_e}(aE + bF)$ be a line bundle on the Hirzebruch surface $\F_e$ with $a\geq -1$. Then:
\begin{enumerate}
\item $h^i(\F_e, \OO_{\F_e}(-E+bF))=0$ for $0 \leq i \leq 2$ and all $b$.
\item $h^0(\F_e, \OO_{\F_e}(bF)) = b+1$ if $b\geq -1$ and $0$ otherwise. In particular, $h^i(\F_e, \OO_{\F_e}(-F))=0$ for $0 \leq i \leq 2$. 
\item We may assume that  $a\geq 1$ and $b \geq 0$. If $b < ae$, then $$h^0(\F_e,  \OO_{\F_e}(aE + bF)) = h^0(\F_e,  \OO_{\F_e}((a-1)E + bF)).$$ If $b\geq ae$, then  $$h^0(\F_e, \OO_{\F_e}(aE + bF))= \chi(\OO_{\F_e}(aE + bF)) = (a+1)(b+1) - e  \frac{a(a+1)}{2}.$$
\end{enumerate}
\end{theorem}
\begin{proof}
We have already observed that $H^i(\OO_{\F_e}(-E+bF))=0$ for $i=0,2$. The case $i=1$ follows from the fact that the Euler characteristic vanishes. This proves (1). 

Next, since $F$ is a pullback from the base, we have  $H^i(\F_e, \OO_{\F_e}(bF)) \cong H^i(\PP^1, \OO_{\PP^1}(b))$. Part (2) of the theorem follows.  

We can therefore assume that $a \geq 1$. If $b <0$, then $h^i(\F_e, \OO_{\F_e}(aE + bF))=0$ for $i=0,2$ and $h^1$ is determined by the Euler characteristic. Hence, we may assume that $a \geq 1$ and $b\geq 0$. If $E \cdot (aE + bF) = b - ae < 0$, then $E$ is in the base locus of the linear system and the map given by multiplication by a section $s_E$ of $\OO_{\F_e}(E)$  $$ H^0(\F_e, \OO_{\F_e}((a-1)E + bF)) \stackrel{s_E}{\longrightarrow}  H^0(\F_e, \OO_{\F_e}(a E + bF))$$ induces an isomorphism. Repeating this process inductively, we reduce to the case when $ae \leq b$. Consider the exact sequence $$0 \longrightarrow \OO_{\F_e}((a-1)E + bF) \longrightarrow \OO_{\F_e}(aE + bF) \longrightarrow \OO_{\PP^1}(b - ae) \longrightarrow 0.$$ If $b-ae \geq -1$, we have a surjection $$H^1(\F_e, \OO_{\F_e}((a-1)E + bF)) \rightarrow  H^1(\F_e, \OO_{\F_e}(aE + bF)) \rightarrow 0.$$ By inductively reducing $a$ to $0$, we conclude that $h^1(\F_e, \OO_{\F_e}(aE + bF))=0$. Consequently, $h^0(\F_e, \OO_{\F_e}(aE + bF))= \chi(\OO_{\F_e}(aE + bF)).$
\end{proof}

\subsection*{Blowups of $\PP^2$}
We next record several basic facts concerning the cohomology of line bundles on blowups of $\PP^2$.  Let $X$ be the blowup of $\PP^2$ at $k$ distinct points $p_1, \dots, p_k$. Let $L$ denote the pullback of the hyperplane class on $\PP^2$ and let $E_i$ denote the exceptional divisor lying over $p_i$. Then $$\Pic(X) \cong \ZZ L \oplus \bigoplus_{i=1}^k \ZZ E_i \ \mbox{with} \ L^2 =1, \ L \cdot E_i =0, \ E_i \cdot E_j = -\delta_{i,j},$$ where $\delta_{i,j}$ is the Kr\"{o}necker delta function. Let $D= \delta L - \sum_{i=1}^k \alpha_i E_i$ be an integral class on $X$. Since $K_X = -3L + \sum_{i=1}^k E_i$, by Riemann-Roch $$\chi(\OO_X(D))= \frac{(\delta+2)(\delta+1)}{2} - \sum_{i=1}^k \frac{\alpha_i (\alpha_i+1)}{2}.$$ If $D$ is effective, then $\delta \geq 0$. Otherwise, a general line with class $L$ would be a moving curve with $L \cdot D < 0$. In particular, by Serre duality, $H^2(X, \OO_X(D))=0$ if $\delta \geq -2$.  

\begin{example}[Del Pezzo surfaces]\label{ex-delPezzo}
Del Pezzo surfaces are smooth complex surfaces $X$ with ample anti-canonical bundle $-K_X$.  They consist of $\PP^1 \times \PP^1$ and the blowup of $\PP^2$ in fewer than 9 points in general position.  Since $\PP^1 \times \PP^1$ is also a Hirzebruch surface, we will concentrate on the surfaces $D_n$, the blowup of $\PP^2$ at $9-n$ general points.  The effective cone of curves on $D_n$ is spanned by the $(-1)$-curves and the nef cone is the dual cone consisting of classes that intersect $(-1)$-curves nonnegatively. The classes of $(-1)$-curves $C= aL - \sum_{i=1}^{9-n} b_i E_i$ on $D_n$ can be obtained by solving the equations $$C^2 = a^2 - \sum_{i=1}^{9-n} b_i^2, \quad -K_{D_n} \cdot C = 3a - \sum_{i=1}^{9-n} b_i =1.$$ For our purposes, it suffices to know that on $D_n$ for $n \geq 3$, the $(-1)$-curves are $E_i$ for $1 \leq i \leq 9-n$, $L-E_i -E_j$ for $i \not= j$ and $2L- E_a - E_b -E_c-E_d-E_e$, where $a,b,c,d,e$ are $5$ distinct indices (whenever the number of points is large enough for these classes to exist)  (see \cite{CoskundelPezzo}, \cite[\S V.4]{Hartshorne}).  
\end{example}

We will need the following cohomology computations. 

\begin{lemma}\label{lem-BasicBP^2}
Let $I \subset \{ 1, \dots, k\}$ be a possibly empty index set. Then:
\begin{enumerate}
\item We have   $H^i(X, \OO_X(D))=0$ for all $i$ if $D$ is one the following
$$-2H + \sum_{i \in I} E_i, \quad -H + \sum_{i\in I} E_i, \quad -E_j + \sum_{i \in I, i \not= j} E_i.$$
\item Assume $H^i(X, \OO_X(D))=0$ for $i>0$. If  $D \cdot E_j \geq 0$  (respectively, $D \cdot L \geq -2$), then $H^i(X, \OO_X(D+E_j))=0$  (respectively, $H^i(X, \OO_X(D+L))=0$) for $i>0$. 
\end{enumerate}
\end{lemma}  

\begin{proof}
If an effective class is represented by a smooth rational curve $C$ on a smooth rational surface $X$, then we claim that $\OO_X(-C)$ has no cohomology. Since $H^i(X, \OO_X) = H^i(C, \OO_C) =0$ for $i \geq 1$,  the natural sequence $$0 \to \OO_X(-C) \to  \OO_X \to \OO_C \to 0$$ implies that $H^i(X, \OO_X(-C))=0$ for all $i$. Since $E_i$, $H$ and $2H$ can be represented by the exceptional curve, a line and a conic, respectively, the proposition is true when $I= \emptyset$. If $D$ is a class such that $H^i(X, \OO_X(D)) =0$ for all $i$ and $E_j \cdot D=0$, then $H^i(X, \OO_X(D + E_j))=0$ for all $i$. To see this, consider the exact sequence $$0 \to \OO_X(D) \to \OO_X(D+E_j) \to \OO_{\PP^1}(-1) \to 0.$$ Since $\OO_X(D)$ and $\OO_{\PP^1}(-1)$ have no cohomology, $\OO_X(D+E_j)$ has no cohomology. Similar sequences imply the last statement. 
\end{proof}

\subsection*{Blowups of Hirzebruch surfaces} Since the blowup of $\F_0$ at one point is isomorphic to the blowup of $\PP^2$ at 2 points and $\F_1$ is isomorphic to the blowup of $\PP^2$ at one point, we may assume that $e \geq 2$.  Let $X$ be the blowup of $\F_e$ along $k$ distinct points $p_1, \dots, p_k$ which are not contained in the exceptional curve $E$. Then the Picard group of $X$ is the free abelian group generated by $E, F, E_1, \dots, E_k$, where $E$ and $F$ are the pullbacks of the two generators from $\F_e$ and $E_1, \dots, E_k$ are the exceptional divisors lying over $p_1, \dots, p_k$. The same argument as in Lemma \ref{lem-BasicBP^2} proves the following.

\begin{lemma}\label{lem-BasicH}
Let $I \subset \{ 1, \dots, k\}$ be a possibly empty index set. Then  $H^i(X, \OO_X(D))=0$ for all $i$ if $D$ is one the following
$$-E + mF + \sum_{i \in I} E_i \quad(m \in \ZZ), \qquad -F + \sum_{i\in I} E_i, \qquad -E_j + \sum_{i \in I, i \not= j} E_i.$$ Moreover, assume $H^i (X, \OO_X(D))=0$ for $i>0$ and $C$ is a rational curve with $C \cdot D \geq -C^2 -1$. Then $H^i (X, \OO_X(D+C))=0$ for $i>0$.
\end{lemma}

\subsection*{Moduli spaces of vector bundles}
Next, we recall some basic facts concerning moduli spaces of Gieseker semistable sheaves and prioritary sheaves. We refer the reader to \cite{CoskunHuizenga},  \cite{Huizenga}, \cite{HuybrechtsLehn} and  \cite{LePotier} for details.

Let $(X, H)$ be a polarized, smooth projective surface. All the sheaves we consider will be pure-dimensional and coherent. If $\cE$ is a pure $d$-dimensional, coherent sheaf, then the Hilbert polynomial has the form
$$P_{\cE}(m) = \chi(\cE(mH)) = a_d \frac{ m^d}{ d!} + \mbox{l.o.t.}$$ The reduced Hilbert polynomial  of $\cE$ is defined by $p_{\cE} = P_{\cE} / a_d$. A sheaf $\cE$ is {\em Gieseker semistable} if for every proper subsheaf $\cF \subsetneq \cE$, we have $p_{\cF} \leq p_{\cE}$, where polynomials are compared for sufficiently large $m$. The sheaf is called {\em Gieseker stable} if for every proper subsheaf the inequality is strict. By theorems of Gieseker, Maruyama and Simpson, there exist projective moduli spaces $M_{X, H} ({\bf v})$ parameterizing $S$-equivalence classes of Gieseker semistable sheaves on $X$ with Chern character ${\bf v}$ (see \cite{HuybrechtsLehn} or \cite{LePotier}). 

It is often hard to verify the stability of a sheaf.  The following notion provides a more flexible alternative.  

\begin{definition}
Let $F$ be a line bundle on $X$. A torsion-free coherent sheaf $\cE$ is {\em $F$-prioritary} if $\Ext^2(\cE, \cE \otimes F^{-1}) =0$. 
\end{definition}

We denote the stack of  $F$-prioritary sheaves on $X$ with Chern character ${\bf v}$ by $\cP_{X, F} ({\bf v})$. The stack $\cP_{X,F} ({\bf v})$ is an open substack of the stack of coherent sheaves. In this paper, we will consider $F$-prioritary sheaves on (blowups of) $\F_{e}$ and  blowups of $\PP^2$, where $F$ is the fiber class on $\F_e$ and the class $L-E_1$ on a blowup of $\PP^2$. The class $F$ endows these surfaces with the structure of a birationally ruled surface.
The following theorem of Walter will be crucial to our arguments.  

\begin{theorem}[{\cite[Proposition 2]{Walter}}]\label{thm-Walter}
Let $X$ be a birationally ruled surface, let $F$ be the fiber class on $X$ and let ${\bf v}$ be a fixed Chern character of rank at least 2. Then the stack $\cP_{X, F}({\bf v})$ of $F$-prioritary sheaves is smooth and irreducible.
\end{theorem}

In particular, if $H$ is an ample divisor on a birationally ruled surface $X$ such that $H \cdot (K_X + F) < 0$ and ${\bf v}$ is a stable Chern character of rank at least 2, then the moduli space $M_{X, H}({\bf v})$ is irreducible and normal \cite[Theorem 1]{Walter}. The inequality $H \cdot (K_X + F) < 0$ guarantees that Gieseker semistable sheaves are $F$-prioritary. When Walter's Theorem applies, we can construct a prioritary sheaf with no cohomology to deduce that general Gieseker semistable sheaves with the same invariants have no cohomology. The advantage is that prioritary sheaves are much easier to construct than semistable sheaves. 

In our computations, we will use the following consequence of Riemann-Roch repeatedly.

\begin{lemma}\label{lem-eulerchar}
Let $\cE$ be a sheaf of rank $r$ on a surface $X$ such that $\chi(\cE)=0$ and let $M$ be a line bundle. Then $$\chi(\cE \otimes M) = \ch_1(E) \cdot \ch_1( M)  + r(\chi(M) - \chi(\OO_X)).$$
\end{lemma}

\begin{proof}
By the Hirzebruch-Riemann-Roch Theorem $$\chi(\cE \otimes M) = \int_X \ch(\cE \otimes M) \td(X) = \int_X \ch(\cE) \ch(M) \td(X).$$ The formula follows immediately by expanding this expression. 
\end{proof}

\section{Strong exceptional collections and resolutions}\label{secResolutions}
In this section, we introduce our first method for proving weak Brill-Noether theorems. This method provides a resolution of the general sheaf of the moduli space in terms of a strong exceptional collection that satisfies certain cohomological properties. The method gives a unirational parameterization of (a component of) the moduli space. The disadvantage is that it is only applicable when a suitable strong exceptional collection exists. We begin by recalling some standard terminology.

\begin{definition}\label{defStrongExceptional} 
A sheaf $A$ is {\em exceptional} if $\Hom(A,A)= \C$ and $\Ext^i(A,A)=0$ for $i \not= 0$. An ordered collection $(A_1, \dots, A_m)$ of exceptional sheaves on a projective variety $X$ is {\em an exceptional collection} if $$\Ext^i(A_t, A_s) =0 \  \mbox{for} \  1 \leq s < t \leq m\  \mbox{and all} \  i.$$ The exceptional collection is {\em strong} if in addition $$\Ext^i(A_s, A_t)=0 \ \mbox{for} \ 1 \leq s < t \leq m \ \mbox{and all}  \ i>0.$$ 
\end{definition}

\begin{example}\label{exSEHirzebruch}
On the Hirzebruch surface $\F_e$, the collection of line bundles
$$\OO_{\F_e}(-E-(e+1)F), \quad \OO_{\F_e}(-E-eF), \quad  \OO_{\F_e}(-F), \quad  \OO_{\F_e}$$ is a strong exceptional collection.  If $1 \leq s < t \leq 4$,  $$\Ext^i(A_t, A_s) \cong H^i(\F_e, -F) \ \mbox{or}  \ H^i(\F_e, -E+b F) \ \mbox{for some} \ b.$$ Since these cohomology groups vanish by Theorem \ref{thmCohomologyH}, we conclude that the collection is exceptional. Similarly, $$\Ext^i(A_s, A_t) \cong H^i(\F_e, F) \ \mbox{or} \ H^i(\F_e, E+b F) \ \mbox{for some} \ b\geq e-1.$$ Since these cohomology groups vanish for $i>0$ by Theorem \ref{thmCohomologyH}, we conclude that the collection is a strong exceptional collection.
\end{example}

\begin{example}\label{exBlowupP2}
Let $\Gamma= \{ p_1, \dots, p_k\}$ be a set of $k$ distinct points on $\PP^2$ and let $X$ be the blowup of $\PP^2$ along $\Gamma$.  Let $E_i$ denote the exceptional divisor lying over $p_i$ and let $L$ be the pullback of the hyperplane class from $\PP^2$. Then
$$\OO_X(-2L), \quad \OO_X(-L), \quad \OO_X(-E_1),\quad  \OO_X(-E_2), \quad \dots \quad \OO_X(-E_k), \quad \OO_{X}$$ is a strong exceptional collection on $X$. This can be checked as follows (see also Bondal's Theorem \cite{Bondal}, \cite{KuleshovOrlov}). Let $I \subset \{1, \dots, k\}$ be an index set. By Lemma \ref{lem-BasicBP^2}, $H^j(X, \OO_X(D))=0$ for all $j$ and $D$  of the form $$-2L + \sum_{i\in I} E_i, \quad  -L + \sum_{i\in I} E_i, \quad -E_i, \quad \mbox{or} \quad  E_l -E_i, l\not= i.$$  Since for $1 \leq s<t\leq m$, each $\Ext^i(A_t, A_s)$ is isomorphic to one of these cohomology groups, we conclude that the collection is exceptional. Similarly,  $H^j(X, \OO_X(D)) =0$ for $j > 0$ and $D$ of the form
$$2L, \quad 2L-E_i, \quad L, \quad L-E_i, \quad E_l - E_i \quad  \mbox{or} \quad E_l.$$
Since the groups $\Ext^i(A_s, A_t)$ for $1 \leq s < t \leq m$ are isomorphic to one of these groups, we conclude that the collection is a strong exceptional collection.
\end{example}

\begin{notation}
Throughout this section, let $X$ be a smooth projective surface and let $(A_1, \dots,  A_m, \OO_X)$ be a strong exceptional collection on $X$. Suppose that a sheaf $\cE$ has a resolution of the form  
\begin{equation}\label{eq-res}
0 \longrightarrow \bigoplus_{i=1}^j A_i^{\oplus a_i} \stackrel{\phi}{\longrightarrow}  \bigoplus_{i=j+1}^{m} A_i^{\oplus a_i} \longrightarrow \cE \longrightarrow 0.
\end{equation}
Notice that $\OO_X$ is the last member of the strong exceptional collection and does not occur in the resolution of $\cE$.
\end{notation}

\begin{lemma}\label{lem-exponents}
Let $\cE$ be a sheaf with a resolution given by Sequence (\ref{eq-res}). Then $H^i(X, \cE)=0$ for all $i$. The exponents $a_i$ are determined by the following relations:
$$a_s = - \chi(\cE, A_s) - \sum_{i=1}^{s-1} a_i \hom(A_i, A_s) \ \  \mbox{for}  \ \ 1 \leq s \leq j  \ \ \mbox{and}$$   $$a_t = \chi(A_t, \cE) - \sum_{i=t+1}^m a_i \hom(A_t, A_i) \ \ \mbox{for} \ \ j+1 \leq t \leq m.$$ 
\end{lemma}

\begin{proof}
Since $(A_1, \dots, A_m, \OO_X)$ is a strong exceptional collection, $\Ext^i(\OO_X, A_s)=0$ for all $1 \leq s \leq m$ and all $i$. Applying $\Ext(\OO_X, -)$ to the Sequence (\ref{eq-res}), we conclude that $$\Ext^i(\OO_X, \cE) = H^i(X, \cE)=0.$$ 
To compute the exponents $a_s$ with $1 \leq s \leq j$, we apply $\Ext(-, A_s)$ to the same sequence. Since $\Ext^k(A_i, A_s)=0$ for all $k$ if $i>s$ and $\Ext^k(A_i, A_s)=0$ for $k > 0$, we obtain the relation
$$\chi(\cE, A_s) + \sum_{i=1}^s a_i \hom(A_i, A_s) = 0.$$ The desired formula follows from the fact that  $\hom(A_s, A_s)=1$. Similarly, to compute the exponents $a_t$ with $j< t <m$, we apply $\Ext(A_t, -)$ to the sequence. We deduce that $$\chi(A_t, \cE) = a_t + \sum_{i=t+1}^m a_i \hom(A_t, A_i).$$ This concludes the proof of the lemma.
\end{proof}

\begin{lemma}\label{lem-Prioritary}
Let $\cE$ be a locally free sheaf with a resolution given by Sequence (\ref{eq-res}) and let $F$ be a line bundle on $X$. Assume that 
\begin{enumerate}
\item $\Ext^1(A_i, A_s \otimes F^{-1})=0$  for $1 \leq i \leq j$ and $j< s \leq m$, and 
\item $\Ext^2 (A_i, A_s \otimes F^{-1})=0$ for $j< i,s \leq m$. 
\end{enumerate}
Then $\cE$ is $F$-prioritary.
\end{lemma}

\begin{proof}
We need to check that $\Ext^2(\cE, \cE\otimes F^{-1})=0$. By applying $\Ext(\cE, -)$ to the sequence $$ 0 \longrightarrow \bigoplus_{i=1}^j A_i^{\oplus a_i}\otimes F^{-1} \longrightarrow  \bigoplus_{i=j+1}^m A_i^{\oplus a_i} \otimes F^{-1}\longrightarrow \cE\otimes F^{-1} \longrightarrow 0,$$ it suffices to check $\Ext^2(\cE, A_s \otimes F^{-1})=0$ for $s>j$. We now apply $\Ext(-, A_s \otimes F^{-1})$ to Sequence (\ref{eq-res}) to obtain  $$\bigoplus_{1 \leq i \leq j} \Ext^1(A_i, A_s \otimes F^{-1})^{\oplus a_i}\longrightarrow \Ext^2(\cE, A_s \otimes F^{-1}) \longrightarrow \bigoplus_{j<i\leq m} \Ext^2(A_i, A_s \otimes F^{-1})^{\oplus a_i}.$$ Since the terms on the right and left are zero by assumption, we conclude that $\cE$ is prioritary with respect to $F$.
\end{proof}

\begin{proposition}\label{lem-complete}
Let $\cE$ be a locally free sheaf with a resolution given by Sequence (\ref{eq-res}). Assume that $\cE$ is $F$-prioritary and has Chern character ${\bf v}$. 
Set $$U= \bigoplus_{i=1}^j A_i^{\oplus a_i} \ \ \mbox{and} \ \ V= \bigoplus_{i=j+1}^m A_i^{\oplus a_i}.$$ Then the open set $S \subset \Hom (U,V)$ parameterizing locally free  $F$-prioritary sheaves is a complete family of $F$-prioritary sheaves. If the stack $\cP_{X, F} ({\bf v})$ is irreducible, then the general sheaf in $\cP_{X, F} ({\bf v})$ has no cohomology. 
\end{proposition}

\begin{proof}
To show that the family is complete, we need to show that the Kodaira-Spencer map is surjective. The Kodaira-Spencer map $$\kappa: T_{\phi} S= \Hom (U,V)  \rightarrow \Ext^1(\cE, \cE)$$  factors as the composition of the two maps \cite{HuybrechtsLehn}, \cite{LePotier}
$$\Hom (U,V) \stackrel{\mu}{\longrightarrow} \Hom(U, \cE ) \stackrel{\nu}{\longrightarrow} \Ext^1(\cE, \cE),$$ where $\mu$ and $\nu$ are the morphisms that appear in the natural exact sequences obtained by applying $\Ext (U, -)$ and $\Ext(-, \cE)$, respectively: $$\Hom(U,V) \stackrel{\mu}{\longrightarrow} \Hom(U, E) \longrightarrow \Ext^1(U,U), \quad \mbox{and} \quad \Hom(U, \cE) \stackrel{\nu}{\longrightarrow} \Ext^1(\cE, \cE) \longrightarrow \Ext^1(V, \cE).$$  Since $(A_1, \dots, A_m, \OO_X)$ is a strong exceptional collection, we have that $$\Ext^1(U,U)=0, \quad \Ext^1(V,V)=0 \quad \mbox{ and} \quad \Ext^2(V,U)=0$$ We conclude that $\mu$ is surjective.  Applying $\Ext(V, -)$ to  Sequence (\ref{eq-res}), we also conclude that 
$$\Ext^1(V, \cE)=0$$ and the map $\nu$ is surjective. Therefore,  the Kodaira-Spenser map is surjective. Since being $F$-prioritary is an open condition, we conclude that we have a complete family of $F$-prioritary sheaves. If $\cP_{X,F}({\bf v})$ is irreducible, then prioritary sheaves having a resolution of the form (\ref{eq-res}) give a Zariski-dense open subset of  $\cP_{X, F}({\bf v})$. By Lemma \ref{lem-exponents}, these sheaves have no cohomology. This concludes the proof of the lemma.
\end{proof}

In view of Proposition \ref{lem-complete}, it is useful to know when $\cE$ given by a resolution of the form (\ref{eq-res}) is locally free. We recall a standard Bertini-type theorem for the reader's convenience (see \cite[Proposition 2.6]{Huizenga}).

\begin{lemma}\label{lem-locallyfree}
Assume that $\sHom(A_s, A_t)$ is globally generated for $1 \leq s \leq j$ and $j < t \leq m$ and that $\rk(\cE) \geq 2$. Then a general sheaf $\cE$ given by a resolution of the form  (\ref{eq-res}) is locally free.
\end{lemma}

\subsection{Weak Brill-Noether for Hirzebruch surfaces}\label{subsecWBNHirzebruch}
As an application  of our discussion, we determine when  weak Brill-Noether holds for Hirzebruch surfaces. Let ${\bf v}$ be a Chern character of rank $r({\bf v})\geq 2$ and Euler characteristic $\chi({\bf v}) = 0$ on a Hirzebruch surface $\F_e$. Let $H$ be an ample divisor on $\F_e$. Since $\F_e$ and $H$ will be fixed in this subsection, we denote $M_{\F_e, H}({\bf v})$ by $M({\bf v})$.  By Walter's Theorem \ref{thm-Walter}, if the moduli space $M ({\bf v})$ is nonempty, then it is irreducible and the general sheaf is locally free.  In particular, Serre duality gives a birational map between $M({\bf v})$ and $M({\bf v}^D)$, where ${\bf v}^D$ is the Serre dual Chern character.  Weak Brill-Noether for $M({\bf v})$ and $M({\bf v}^D)$ are equivalent problems.

We write $$\nu({\bf v}) = \frac{c_1({\bf v})}{r({\bf v})} = \frac{k}{r} E + \frac{l}{r}F$$ for the ``total slope'' of ${\bf v}$, so $\mu_H({\bf v}) = \nu({\bf v})\cdot H$.  Since $$K_{\F_e} = -2E - (2+e)F,$$ we may replace ${\bf v}$ by ${\bf v}^D$ if necessary to assume $$\frac{k}{r}\geq -1.$$  Furthermore, if $\frac{k}{r}=-1$ we may additionally assume $\frac{l}{r} \geq -1-\frac{e}{2}$.   The Bogomolov inequality gives a further restriction on $\nu({\bf v})$.

\begin{lemma}\label{lem-Bogomolov}
With ${\bf v}$ as above, if $M({\bf v})$ is nonempty, then $\frac{l}{r} \geq -1+\frac{ke}{2r}$.
\end{lemma}
\begin{proof}
There is nothing to prove in the case $\frac{k}{r} = -1$, since then we are already assuming $\frac{l}{r} \geq -1-\frac{e}{2}$.  Assume $\frac{k}{r}>-1$ and $\frac{l}{r}<-1+\frac{ke}{2r}$.  We show $M({\bf v})$ is empty.  Let $$P(\nu) = \chi(\OO_{\F_e}) + \frac{1}{2}(\nu^2-\nu\cdot K_{\F_e}),$$ so that the Riemann-Roch formula takes the form $$\chi({\bf v}) = r(P(\nu({\bf v}))-\Delta({\bf v})).$$  Then since $\chi({\bf v}) = 0$, we have  $$\Delta({\bf v}) = P(\nu({\bf v})) = \left(\frac{k}{r}+1\right)\left(\frac{l}{r}+1\right)-\frac{1}{2}e\left(\frac{k}{r}+1\right)\left(\frac{k}{r}\right)= \left(\frac{k}{r}+1\right)\left(\frac{l}{r}+1-\frac{ek}{2r}\right)<0.$$ Thus, by the Bogomolov inequality, $M({\bf v})$ is empty.
\end{proof}

Our assumptions on ${\bf v}$ now give some simple $H^2$-vanishing results for semistable sheaves.

\begin{lemma}\label{lem-h2}
With ${\bf v}$ as above, if $\cE$ is an $H$-semistable sheaf of character ${\bf v}$, then $$H^2(\F_e, \cE) = H^2(\F_e, \cE(-E))=0.$$
\end{lemma}
\begin{proof}
Since $\frac{k}{r}\geq -1$, Lemma \ref{lem-Bogomolov} in particular implies $\frac{l}{r}\geq -1 -\frac{e}{2}$.  In fact, we prove the stronger result that if ${\bf v}\in K(\F_e)$ is any character such that $\frac{k}{r} \geq -2$, $\frac{l}{r}\geq -2-e$, and at least one of the inequalities is strict, then any $H$-semistable sheaf $\cE$ of character ${\bf v}$ has $H^2(\F_e, \cE)=0$.

We use Serre duality to write $$H^2(\F_e, \cE) = \Ext^2(\OO_{\F_e},\cE) = \Hom(\cE,K_{\F_e})^*.$$ Since $\cE$ and $K_{\F_e}$ are both $H$-semistable, the vanishing will follow if $\mu_H({\cE})> \mu_H(K_{\F_e})$, which is equivalent to $H\cdot(\nu({\bf v})-K_{\F_e})>0$. The nef cone of $\F_e$ is spanned by $F$ and $E+eF$.  We compute \begin{align*}F\cdot(\nu({\bf v})-K_{\F_e})&=\frac{k}{r}+2\\
(E+eF)\cdot (\nu({\bf v})-K_{\F_e})&= \frac{l}{r}+2+e.\end{align*}
By our assumption on ${\bf v}$, both of these intersection numbers are nonnegative, and at least one of them is positive.  Since $H$ is ample, it is a positive combination of the extremal nef classes, and it follows that $H\cdot (\nu({\bf v})-K_{\F_e})>0$.  
\end{proof}

Next we describe characters ${\bf v}$ such that weak Brill-Noether fails for $M({\bf v})$ if the moduli space is nonempty.  It will turn out that these are the only characters where weak Brill-Noether fails.

\begin{proposition}
Let ${\bf v}$ be a character as above, and suppose $\cE$ is an $H$-semistable sheaf of character ${\bf v}$.  If $\chi(\cE(-E))>0$, then $H^0(\F_e, \cE)\neq 0$.  More explicitly, if $$\frac{l-ke}{r} = \nu({\bf v})\cdot E< -1,$$ then $H^0(\F_e, \cE)\neq 0$.
\end{proposition}
\begin{proof}
The restriction sequence $$0\to \cE(-E)\to \cE\to \cE|_E\to 0$$ shows that if $H^0(\F_e, \cE(-E))\neq 0$, then $H^0(\F_e, \cE) \neq 0$.  By Lemma \ref{lem-h2}, we have $H^2(\F_e, \cE(-E))=0$. Therefore, $H^0(\F_e, \cE(-E))\neq 0$ since $\chi(\cE(-E))>0$. 
The second statement follows immediately from Riemann-Roch.
\end{proof}

Conversely, we have the main result of this section.

\begin{theorem}\label{thmMainH}
Let ${\bf v}$ be a character as above, and suppose $\cE$ is a general $H$-semistable sheaf of character ${\bf v}$.  If $\chi(\cE(-E))\leq 0$, then $\cE$ has no cohomology.  

Moreover, unless $\cE$ is a direct sum of copies of $\OO_{\P^1\times \P^1}(-1,-1)$, $\cE$ admits a resolution of the form $$0\to \OO_{\F_e}(-E-(e+1)F)^{a}\to \OO_{\F_e}(-E-eF)^{b}\oplus \OO_{\F_e}(-F)^{c}\to \cE\to 0$$ for some nonnegative numbers $a,b,c$.
\end{theorem}
\begin{proof}
We may assume $\cE$ is not a direct sum of copies of $\OO_{\P^1\times \P^1}(-1,-1)$, since there is nothing to prove in that case.  By Rudakov's classification, $\cE$ is a direct sum of copies of $\OO_{\P^1\times \P^1}$ if and only if $e=0$ and $k = l = -r$, so we will assume these equalities do not all hold.

By Example \ref{exSEHirzebruch}, the collection 
$$A_1= \OO_{\F_e}(-E-(e+1)F), \quad A_2= \OO_{\F_e}(-E-eF), \quad A_3= \OO_{\F_e}(-F), \quad \OO_{\F_e}$$ is a strong exceptional collection. Let $\cE$ be the cokernel of a general map
$$0\to \OO_{\F_e}(-E-(e+1)F)^{a} \stackrel{\phi}{\longrightarrow} \OO_{\F_e}(-E-eF)^{b}\oplus \OO_X(-F)^{c}\to \cE\to 0.$$

\noindent \emph{Step 1: Nonnegativity of the exponents.} First, we determine the exponents $a,b,c$ that give the correct Chern class for $\cE$ and we verify that they are nonnegative.  

By Lemma \ref{lem-exponents}, $c = \chi(\cE(F)).$ By Lemma \ref{lem-eulerchar},  we get $$c=\chi(\cE(F))=c_1(\cE)\cdot F +r=k+r\geq 0$$ since $\frac{k}{r}\geq -1$. Hence,  $c\geq 0$. 

Next, applying $\chi(\OO_X(E),-)$, we get $$b = -\chi(\OO_X(E),\cE)=-\chi(\cE(-E)).$$ Our assumption $\chi(\cE(-E))\leq 0$ yields $b\geq 0$.

The nonnegativity of the exponent $a$ is the most challenging.  By Lemmas \ref{lem-eulerchar} and \ref{lem-exponents}, $$a=-\chi(\cE(-E-F))= l -ke +k+r.$$ By assumption, $$ l-ke+r \geq 0 \quad \mbox{and} \quad l-\frac{ke}{2} + r \geq 0.$$ Hence, if either $k \geq 0,$ or $k < 0$ and $e \geq 2$, then $a\geq 0$. 

It remains to consider the cases when $k<0$ and $e=0$ or $1$.  Suppose $a<0$ to get a contradiction.  Since we are assuming $k\geq -r$, we must have $l <0$.  We claim that $$\Ext^2(\OO(E+F), \cE)= \Hom(\cE, \OO(-E-(e+1)F))^*=0$$ by stability. To see this compare the $H=E +\alpha F$ slopes (with $\alpha >e$)
$$\mu_H(\cE) = \frac{k}{r} \alpha + \frac{l}{r} - \frac{ke}{r} \geq   \frac{k}{r} \alpha + \frac{l}{r} - \frac{ke}{2r} \geq -\alpha -1 = \mu_H (\OO(-E - (e+1)F)).$$ The first inequality is strict unless $e=0$.  If $e=0$, then the second inequality is strict unless $k=l=-r$.  Since we are assuming $\cE$ is not a direct sum of copies of $\OO_{\P^1\times \P^1}(-1,-1)$, we conclude that the inequality is always strict, and the vanishing $\Ext^2(\OO(E+F),\cE)=0$ holds.  Then since $a$ is negative, we deduce $\Hom(\OO(E+F),\cE)\neq 0$.  This contradicts the $H$-semistability of $\cE$, since $\mu_H(\OO(E+F)) > 0 > \mu_H(\cE)$.  Therefore $a \geq 0$ in every case.

\noindent \emph{Step 2: A general sheaf with the specified resolution is locally free of class ${\bf v}$.} Since $$\sHom(\OO_{\F_e}(-E-(e+1)F), \OO_{\F_e}(-E-eF)) \cong \OO_{\F_e}(F) \quad \mbox{and}$$   $$\sHom(\OO_{\F_e}(-E-(e+1)F), \OO_{\F_e}(-F)) \cong \OO_{\F_e}(E+eF)$$ are globally generated, by Lemma \ref{lem-locallyfree}, the cokernel of a general map $\phi$ is locally free.  

\noindent \emph{Step 3: Any locally free sheaf $\cE$ with the specified resolution is prioritary.} By Theorem \ref{thmCohomologyH}, $$\Ext^1(A_1, A_2 \otimes\OO_{F_e}(-F)) = H^1(\F_e, \OO_{\F_e})=0,$$ $$\Ext^1(A_1, A_3  \otimes\OO_{F_e}(-F)) = H^1(\F_e, \OO_{\F_e}(E+(e-1)F))=0$$ and $$\Ext^2(A_2, A_2 \otimes\OO_{F_e}(-F)) = \Ext^2(A_3, A_3\otimes\OO_{F_e}(-F)) = H^2(\F_e, \OO_{\F_e}(-F))=0,$$ $$\Ext^2(A_2, A_3 \otimes\OO_{F_e}(-F)) = H^2(\F_e, \OO_{\F_e}(E + (e-2)F))=0.$$ By Lemma \ref{lem-Prioritary}, we conclude that $\cE$ is $F$-prioritary.

\noindent \emph{Step 4: Conclusion of the proof.} By Proposition \ref{lem-complete}, the open subset $S \subset \Hom(A_1^a,A_2^b\oplus A_3^c)$ parameterizing locally free quotients parameterizes a complete family of prioritary sheaves.  Any sheaf parameterized by $S$ has no cohomology by Lemma \ref{lem-exponents}. Since the stack $\cP_{\F_e, F}({\bf v})$ of prioritary sheaves of character ${\bf v}$ is irreducible and $\cM ({\bf v})$ is a dense open substack of $\cP_{\F_e, F}({\bf v})$, we conclude that weak Brill-Noether holds for $M({\bf v})$. Furthermore, the general sheaf in $M({\bf v})$ has the required resolution.  In particular, note that  the moduli space is unirational since it is dominated by an open set in $\Hom(A_1^a,A_2^b\oplus A_3^c)$.
\end{proof}

\subsection{Applications to blowups of $\PP^2$}
In this subsection, let $\Gamma$ be a set of $k$ distinct points $p_1, \dots, p_k$ on $\PP^2$ and let $X$ denote the blowup of $\PP^2$ along $\Gamma$. Let $L$ denote the pullback of the hyperplane class and let $E_i$ denote the exceptional divisor lying over $p_i$.  Our methods have the following  consequence. 

\begin{theorem}\label{thm-veryweak}
Let ${\bf v}$ be a  Chern character on $X$ such that $$r({\bf v})\geq 2, \quad \chi({\bf v})=0, \quad   \nu({\bf v}) = \delta L - \sum_{i=1}^k \alpha_i E_i, \  \mbox{where} \ \delta, \alpha_i \geq 0\ \mbox{and} \   \delta - \sum_{i=1}^k \alpha_i \geq -1.$$ Then the stack $\cP_{X, L-E_1}({\bf v})$ is nonempty, and a general sheaf parameterized by $\cP_{X,L-E_1}({\bf v})$ has no cohomology.
  
If $\delta - 2\sum_{i=1}^k \alpha_i +2 \geq 0$, then the general sheaf $\cE$ in $\cP_{X, L-E_1}({\bf v})$  has a resolution of the form
\begin{equation}\label{eqfirst}
0 \longrightarrow  \OO_X(-2L)^a \stackrel{\phi}{\longrightarrow } \OO_X(-L)^b \oplus \bigoplus_{i=1}^k \OO(-E_i)^{c_i} \longrightarrow \cE \longrightarrow 0, 
\end{equation}
Otherwise, $\cE$ has a resolution of the form
\begin{equation}\label{eqsecond}
0 \longrightarrow  \OO_X(-2L)^a \oplus \OO_X(-L)^b \stackrel{\phi}{\longrightarrow}   \bigoplus_{i=1}^k \OO(-E_i)^{c_i} \longrightarrow \cE \longrightarrow 0.
\end{equation}
The exponents are given by  $$a = r({\bf v}) (\delta - \sum_{i=1}^k \alpha_i +1), \quad c_i = r({\bf v}) \alpha_i, \quad b= \left|r({\bf v})+a - \sum_{i=1}^k c_i\right|.$$ 
\end{theorem}

\begin{proof}
The linear system $|L-E_1|$ defines a map from $X$ to $\PP^1$ giving $X$ the structure of a birationally ruled surface.  By \cite[Proposition 2]{Walter}, the stack of prioritary sheaves $\cP_{X, L-E_1}({\bf v})$ is smooth and irreducible. 
By Example \ref{exBlowupP2},  $$A_1 = \OO(-2L), \quad A_2= \OO(-L), \quad A_3= \OO(-E_1), \dots, A_{k+2}=\OO(-E_k), \quad \OO$$  is a strong exceptional collection on $X$. Lemma \ref{lem-exponents} computes the exponents of a sheaf with Chern character ${\bf v}$ with the given resolutions (\ref{eqfirst}) or (\ref{eqsecond}). Our assumptions on ${\bf v}$ imply that these exponents are positive; the requirement on the sign of $\delta - 2\sum_{i=1}^k \alpha_i+2$ which is used to determine the form of the resolution ensures that the exponent $b$ is positive.  Lemma \ref{lem-exponents} also implies that sheaves with these resolutions have no cohomology. 

Since $$\sHom(\OO_X(-2L), \OO_X(-L)) \cong \OO_X(L), \quad \sHom(\OO_X(-2L), \OO_X(-E_l)) \cong \OO_X(2L-E_i)$$ and $$ \sHom(\OO_X(-L), \OO_X(-E_i)) = \OO_X(L-E_i)$$ are globally generated and the rank of ${\bf v}$ is at least 2, by Lemma \ref{lem-locallyfree}, the cokernel $\cE$ of a general map $\phi$ is locally free in both cases. 

By Lemma \ref{lem-BasicBP^2}, $H^1(X, \OO_X(D))=0$ for $D$  a divisor of the form
$$E_1, \quad L, \quad L +E_1-E_i, \quad \emptyset, \quad E_1 - E_i.$$ Hence, $$\Ext^1(A_1, A_l \otimes \OO_X(-L+E_1))=0 \ \mbox{for} \ 2 \leq l \leq k+2$$ and $$\Ext^1(A_2, A_l\otimes \OO_X(-L+E_1))=0 \mbox{ for} \  3 \leq l \leq k+2.$$ Similarly, $H^2(X, \OO_X(D))=0$ when $D$ is an integral of the form $d  L - \sum a_i E_i$ with $d \geq -2$. Hence, $$\Ext^2(A_i, A_l \otimes \OO_X(-L+E_i))=0 \ \mbox{for}  \ 2 \leq i, l \leq k+2.$$ By Lemma \ref{lem-Prioritary}, all  locally-free sheaves with resolutions given by (\ref{eqfirst}) or (\ref{eqsecond}) are prioritary with respect to $\OO(L-E_1)$.  We conclude that the stack of prioritary sheaves $\cP_{X, L-E_1}({\bf v})$ is nonempty. Finally, by Lemma \ref{lem-complete}, this is a complete family of prioritary sheaves. Since the stack $\cP_{X, L-E_1}({\bf v})$ is irreducible, the general prioritary sheaf with Chern character ${\bf v}$ has no cohomology and has a resolution of the form (\ref{eqfirst}) or (\ref{eqsecond}) depending on the sign of the exponent of $\OO_X(-L)$.
\end{proof}

\begin{corollary}\label{cor-veryweak}
Let $H$ be an ample divisor on $X$ such that $H \cdot (-2L + \sum_{i=2}^k E_i) <0$. Let ${\bf v}$ be a stable Chern character on $X$ satisfying the hypotheses of Theorem \ref{thm-veryweak}. Then the moduli space $M_{X, H}({\bf v})$ is unirational and satisfies  weak Brill-Noether.
\end{corollary}
\begin{proof}
By Walter \cite{Walter}, an $H$-semistable sheaf is prioritary with respect to $L-E_1$. Consequently, $\cM_{X, H}({\bf v})$ is an open substack of $\cP_{X, L-E_1}({\bf v})$. The corollary follows from Theorem \ref{thm-veryweak}.
\end{proof}

\begin{remark}
Let $X$ be the blowup of $\PP^2$ along $k$ collinear points. Let $\cE$ be a coherent sheaf with Chern character ${\bf v}$ with  $$\mu({\bf v}) =\delta L - \sum_{i=1}^k \alpha_i E_i, \quad \delta - \sum_{i=1}^k \alpha_i < -1.$$ Assume $\cE$ is semistable with respect to an ample divisor $H$ such that $$\mu(\cE) \cdot H >  -2L \cdot H.$$ Then we claim that the cohomology of $\cE$ does not vanish and $M_{X,H}({\bf v})$ does not satisfy weak Brill-Noether. By stability,
$$\Ext^2(\OO(L-\sum_{i=1}^k E_i), \cE))= \Hom (\cE, \OO(-2L))^*=0.$$ By Lemma \ref{lem-eulerchar}, $$\chi( \OO(L-\sum_{i=1}^k E_i), \cE) = \rk(\cE)(-\delta + \sum_{i=1}^k \alpha_i-1) >0.$$ We conclude that $$\Hom(\OO(L-\sum_{i=1}^k E_i), \cE) \not=0.$$ However, since the $k$ points  are collinear, composing the natural map $\Hom(\OO, \OO(L - \sum_{i=1}^k E_i))$ with a nonzero morphism to $\cE$, we see that $H^0(X, \cE)\not= 0$. 
Hence, without further assumptions on the positions of the points, Theorem \ref{thm-veryweak} is sharp. However, if we assume that the points are general, we can use different strong exceptional collections to extend the range where  weak Brill-Noether holds.  
\end{remark}

\subsection{Applications to blowups of Hirzebruch surfaces}
A very similar theorem holds for blowups of Hirzebruch surfaces. Let $X$ be the blowup of a Hirzebruch surface $\F_e$, $e\geq 1$, along $k$ distinct points not lying on the exceptional curve $E$. Denote by $E$ and $F$ the pullbacks of the corresponding classes on $\F_e$ and let $E_1, \dots, E_k$ denote the exceptional divisors lying over the points $p_1, \dots, p_k$.  Let $\cE$ be a locally sheaf of rank at least $2$ with total slope
$$\nu(\cE) := \frac{\ch_1(\cE)}{\rk(\cE)}  = \alpha E + \beta F - \sum_{i=1}^k \alpha_i E_i.$$

\begin{theorem}\label{thm-veryweakH}
Let ${\bf v}$ be a Chern character on $X$ such that  $\nu({\bf v})$ satisfies $$\alpha_i \geq 0, \quad \alpha - \sum_{i=1}^k \alpha_i \geq -1, \quad \mbox{and} \quad  \beta - \sum_{i=1}^k \alpha_i   +1 \geq \max((e-1) \alpha, e\alpha).$$ Then the stack $\cP_{X, F}({\bf v})$ is nonempty and the general sheaf in $\cP_{X, F}({\bf v})$ has no cohomology. Furthermore, the general sheaf in $\cP_{X, F}({\bf v})$ admits a resolution of the form 
$$
0\to \OO_X(-E-(e+1)F)^{a}\stackrel{\phi}{\to} \OO_X(-E-eF)^{b}\oplus \OO_X(-F)^{c}\oplus \bigoplus_{i=1}^k \OO_X(-E_i)^{d_i} \to \cE\to 0, 
$$
where $$a= r({\bf v})(\beta - (e-1) \alpha - \sum_{i=1}^k \alpha_i +1), \quad b= r({\bf v})(\beta- e\alpha - \sum_{i=1}^k \alpha_i +1),$$  $$c = r({\bf v})(\alpha - \sum_{i=1}^k \alpha_i +1) \quad d_i = r({\bf v}) \alpha_i.$$
\end{theorem}

\begin{proof}
Since the proof of this theorem is similar to the proof of Theorems \ref{thmMainH} and \ref{thm-veryweak}, we leave the routine verifications to the reader.  By Lemma \ref{lem-BasicH}, the sequence
$$\OO_X(-E-(e+1)F), \quad \OO_X(-E-eF) \quad \OO_X(-F), \quad \OO_X(-E_1), \dots, \OO_X(-E_k), \quad \OO_X$$ is a strong exceptional collection. By Lemma \ref{lem-exponents}, the exponents are as claimed and are positive by assumption. Lemma \ref{lem-locallyfree} applies to show that the cokernel of a general map $\phi$ is locally free. By Lemma \ref{lem-BasicH}, Lemma \ref{lem-Prioritary} applies. By Lemmas \ref{lem-Prioritary} and \ref{lem-complete}, one obtains a complete family of $F$-prioritary sheaves. Since the stack of $F$-prioritary sheaves is irreducible, the theorem follows.
\end{proof}

As usual, we obtain the following corollary.

\begin{corollary}
Let $H$ be an ample divisor on $X$ such that $H \cdot (-2E - (e+1)F + \sum_{i=1}^k E_i )< 0$. Let ${\bf v}$ be a stable Chern character satisfying the assumptions of Theorem \ref{thm-veryweakH}. Then $M_{X,H}({\bf v})$ is unirational and satisfies weak Brill-Noether. 
\end{corollary}

\section{Blowups of $\P^2$ revisited}\label{secPrioritary}

Let $X = \Bl_{p_1,\ldots,p_k} \P^2$ be the blowup of $\P^2$ at $k$ distinct points $p_1,\ldots,p_k\in \P^2$.  In this section we study the weak Brill-Noether problem for $X$.  In particular, we give some sufficient conditions on a character ${\bf v}$ of Euler characteristic $0$ for weak Brill-Noether to hold for certain moduli spaces of sheaves of character ${\bf v}$.  We will prove sharper results when $X$ is a del Pezzo surface of degree at least $4$ in the next section, using the basic tools developed in this section as a starting point. Fix $F = L-E_1$, so that the complete series $|F|$ induces a map $X\to \P^1$ with fiber class $F$, exhibiting $X$ as a birationally ruled surface.

If $H$ is an ample divisor such that $H \cdot (K_X+F) < 0$, then any $H$-semistable sheaf is automatically $F$-prioritary, so that $\cP_{X,F}({\bf v})$ contains the stack $\cM_{X,H}({\bf v})$ of $H$-semistable sheaves as a dense open substack whenever $H$-semistable sheaves of character ${\bf v}$ exist.  Note that such polarizations $H$ always exist; since $F \cdot (K_X+F) = -2$, any ample divisor $H$ spanning a ray sufficiently close to the ray spanned by the nef divisor $F$ will do the trick.  In what follows we work primarily with prioritary sheaves instead of semistable sheaves. Prioritary sheaves have the advantage that they are much easier to construct than semistable sheaves.

\subsection{Prioritary direct sums of line bundles} 
By Walter's irreducibility theorem, in order to prove a general prioritary sheaf of some character ${\bf v}$ has no cohomology, it suffices to construct a particular such sheaf.  It is often possible to do this by considering elementary modifications of direct sums of line bundles.  First we give a criterion for determining when a direct sum of line bundles is prioritary. 
\begin{lemma}\label{lem-prioritary}
Let $\cE = L_1\oplus \cdots \oplus L_r$ be a direct sum of line bundles.  Suppose that $N$ is a nef divisor such that $$N \cdot (L_i-L_j) < -N \cdot (F+K_X)$$ for all $i,j$.  Then $\cE$ is $F$-prioritary.
\end{lemma}
\begin{proof}
The vanishing $\Ext^2(\cE,\cE(-F))= 0$ will follow if $\Ext^2(L_i,L_j(-F)) =0$ for all $i,j$.  We have $$\Ext^2(L_i,L_j(-F)) \cong H^2( -L_i + L_j -F) \cong H^0(L_i-L_j+F+K_X)^*.$$ Then $$N \cdot (L_i-L_j+F+K_X)<0,$$ and since $N$ is nef we conclude $L_i-L_j+F+K_X$ is not effective.    Therefore $\Ext^2(L_i,L_j(-F))=0$.
\end{proof}

Further prioritary sheaves of lower Euler characteristic can be constructed by elementary modifications.

\begin{lemma}\label{lem-elementary}
Let $\cE$ be an $F$-prioritary sheaf on $X$, and let $p\in X$ be a point where $\cE$ is locally free.  Pick a surjection $\cE\to \OO_p$, and consider the elementary modification $\cE'$ defined by the sequence $$0\to \cE'\to \cE\to \OO_p \to 0.$$ Then $\cE'$ is $F$-prioritary with $\chi(\cE') = \chi(\cE)-1$.  If furthermore $h^0(\cE) >0$ and $p$ and the map $\cE\to \OO_p$ are general, then additionally $h^0(\cE') = h^0(\cE) - 1$.
\end{lemma}
\begin{proof}
Clearly $\cE'$ is torsion-free since $\cE$ is.  Then $\Ext^2(\cE',\cE'(-F))$ is a quotient of $\Ext^2(\cE,\cE'(-F))$, so it suffices to prove the latter group vanishes.  Tensoring the exact sequence by $\OO_X(-F)$ and applying $\Hom(\cE,-)$, we get an exact sequence $$\Ext^1(\cE,\OO_p)\to \Ext^2(\cE,\cE'(-F))\to \Ext^2(\cE,\cE(-F)).$$ Then $\Ext^2(\cE,\cE(-F))=0$ since $\cE$ is prioritary and $\Ext^1(\cE,\OO_p)=0$ since $\cE$ is locally free at $p$.
If $h^0(\cE) >0$ and $p$ is general, then $\cE$ has a section which does not vanish at $p$.  Therefore a general surjection $\cE\to \OO_p$ induces a surjective map $H^0(\cE) \to H^0(\OO_p)$, and we conclude $h^0(\cE') = h^0(\cE)-1$.
\end{proof}

These two lemmas motivate the next definition.

\begin{definition}\label{def-good}
Fix a nef divisor $N$ such that $-N \cdot (F+K_X) \geq 2$.  We call a direct sum $\cE = L_1\oplus \cdots \oplus L_r$ of line bundles ($N$-)\emph{good} if the following properties are satisfied.
\begin{enumerate}
\item $\cE$ has no higher cohomology: $h^i(\cE) = 0$ for $i>0$.  
\item For all $i,j$ we have $N \cdot (L_i-L_j)\leq 1$.  In particular, $\cE$ is $F$-prioritary by Lemma \ref{lem-prioritary}.
\end{enumerate}
Fix an integer $r\geq 1$.  We let $$\Lambda_{r,N} = \{c_1(\cE) : \textrm{$\cE$ is a rank $r$ good direct sum of line bundles}\}\subset N^1(X)_\Z.$$
 \end{definition}

The next result follows immediately from Lemma \ref{lem-elementary} and the definitions.

\begin{corollary}\label{cor-latticeVanish}
Suppose ${\bf v}\in K(X)$ has $\chi({\bf v}) = 0$ and $r=r({\bf v}) >0$, and fix a nef divisor $N$ as in Definition \ref{def-good}.  If $c_1({\bf v})\in \Lambda_{r,N}$, then $\cP_{X,F}({\bf v})$ is nonempty and a general sheaf $\cE\in \cP_{X,F}({\bf v})$ has no cohomology.
\end{corollary}

Our next result is our strongest result on the weak Brill-Noether problem for an arbitrary blowup of $\P^2$.  

\begin{theorem}\label{thm-blowup}
Let $X = \Bl_{p_1,\ldots,p_k} \P^2$ be a blowup of $\P^2$ at $k$ distinct points.  Let ${\bf v}\in K(X)$ with $r=r({\bf v})>0$, and write $$\nu({\bf v}) := \frac{c_1({\bf v})}{r({\bf v})} = \delta L - \alpha_1E_1-\cdots -\alpha_k E_k,$$ so that the coefficients $\delta,\alpha_i\in \Q$.  Assume that $\delta\geq 0$ and $\alpha_i \geq 0$ for all $i$.  Suppose that the line bundle $$\lfloor \delta \rfloor L - \lceil \alpha_1\rceil E_1-\cdots - \lceil\alpha_k\rceil E_k$$ has no higher cohomology.  Then choosing $N = L$, we have $c_1({\bf v}) \in \Lambda_{r,L}$.  In particular, if $\chi({\bf v}) = 0$, then $\cP_{X,F}({\bf v})$ is nonempty and the general $\cE\in \cP_{X,F}({\bf v})$ has no cohomology.
\end{theorem}
\begin{proof}
In more detail, write $$\nu({\bf v}) = \left(d+\frac{p}{r}\right)L - \left(a_1+\frac{p_1}{r}\right)E_1-\cdots- \left(a_k+\frac{p_k}{r}\right)E_k$$ where $d = \lfloor \delta\rfloor$ and $a_i = \lfloor \alpha_i\rfloor.$ Suppose $M = eL - b_1E_1-\cdots-b_kE_k$ is a line bundle such that $e\in \{d,d+1\}$ and $b_i\in \{a_i,a_i+1\}$ for all $i$.  Our assumptions imply that $M$ has no higher cohomology.  Now we can construct a direct sum $\cE = L_1\oplus \cdots \oplus L_r$ of line bundles of this form such that $c_1(\cE) = c_1({\bf v})$.  Indeed, we only need to ensure that exactly $p$ of the $L_i$'s have a coefficient of $d+1$ on $L$, and similarly for the other coefficients.  By construction, $\cE$ is $L$-good.
\end{proof}

\section{Del Pezzo surfaces}\label{secDelPezzo}

In this section we improve on Theorem \ref{thm-blowup} in the special case of a smooth del Pezzo surface $X$ of degree $4\leq d \leq 7$.  Thus $X = \Bl_{p_1,\ldots,p_k} \P^2$ is a blowup of $\P^2$ at $2\leq k = 9-d \leq 5$ points with no three lying on a line, and $-K_X = 3L - E_1-\cdots -E_k$ is ample.  As in the previous section we fix a fiber class $F = L-E_1$ and study $F$-prioritary sheaves.  Our main theorem is the following.

\begin{theorem}\label{thm-delPezzoMain}
Let $X$ be a del Pezzo surface of degree $4\leq d \leq 7$.  Let ${\bf v}\in K(X)$ with $\chi({\bf v}) = 0$, and suppose $c_1({\bf v})$ is nef.  Then the stack $\cP_{X,F} ({\bf v})$ of prioritary sheaves is nonempty and a general $\cE\in \cP_{X, F}({\bf v})$ has no cohomology.
\end{theorem}

The cone of curves $\NE(X)$ is spanned by the classes of the $(-1)$-curves on $X$; the $(-1)$-curves are an exceptional divisor, a line through two points, or a conic through $5$ points if $d=4$ (review Example \ref{ex-delPezzo}).  Dually, the nef cone $\Nef(X)\subset N^1(X)$ is the subcone of classes $\nu$ such that $\nu \cdot C \geq 0$ for every $(-1)$-curve $C$.  

The main additional ingredient that goes into the proof of Theorem \ref{thm-delPezzoMain} is the action of the Weyl group on $\Pic(X)$.  The Weyl group preserves the intersection pairing, so it preserves the nef cone, the Euler characteristic, and the dimensions of cohomology groups of line bundles.  The Weyl group acts transitively on length $\ell\leq k$ configurations $(C_1,\ldots,C_\ell)$  of disjoint $(-1)$-curves unless $\ell = k-1$, in which case the action has two orbits \cite{Manin}.  

The Weyl group additionally preserves the canonical bundle $K_X$.  For this reason, we will take $N=-K_X$ in Definition \ref{def-good} and study $(-K_X)$-good direct sums of line bundles---note that the needed inequality $K_X.(F+K_X) \geq 2$ follows from our assumption that $d\geq 4$.  Then the Weyl group additionally preserves the set $\Lambda_{r,-K_X}$ of first Chern classes of rank $r$ $(-K_X)$-good direct sums of line bundles.  Theorem \ref{thm-delPezzoMain} is a direct consequence of the next result and Corollary \ref{cor-latticeVanish}.

\begin{proposition}\label{prop-delPezzoContained}
If $X$ is a del Pezzo surface of degree $4\leq d\leq 7$ we have $\Nef(X)_\Z\subset \Lambda_{r,-K_X}$.
\end{proposition}

Since the Weyl group preserves both the nef cone and $\Lambda_{r,-K_X}$, we only need to show that if $D\in \Nef(X)_\Z$, then some translate of $D$ by a Weyl group element is in $\Lambda_{r,-K_X}$.  Our next lemma further allows us to repeatedly replace $D$ with a different divisor $D'$ until $D'$\ is on the boundary of the nef cone.  

\begin{lemma}\label{lem-upshift}
Let $D\in \Nef(X)_\Z$ be a nef divisor, and suppose $D \cdot E_i \geq D \cdot E_j$.  Suppose $D' = D - E_i + E_j$ is nef and that $D'\in \Lambda_{r,-K_X}$.  Then $D\in \Lambda_{r,-K_X}$.
\end{lemma}
\begin{proof}
Since $D'\in \Lambda_{r,-K_X}$, there is a $(-K_X)$-good direct sum of line bundles $\cE'$ such that $c_1(\cE') = D'$.  Since $D \cdot E_i\geq D \cdot E_j$, we have $D' \cdot E_i > D' \cdot E_j$, and at least one of the line bundles $L'$ in $\cE'$ must have $L'.E_i > L'.E_j \geq -1$.  If we put $M = L'+E_i-E_j$, then since $L'$ has no higher cohomology it follows that $M$ has no higher cohomology.  Additionally, $M \cdot (-K_X) = L' \cdot (-K_X)$.  Then if we define $\cE$ by taking $\cE'$ and replacing $L'$ with $M$, it follows that $\cE$ is $(-K_X)$-good and $c_1(\cE) = D$.
\end{proof}

Now let $D\in \Nef(X)_\Z$ be arbitrary.  If there are $i,j$ such that Lemma \ref{lem-upshift} can be applied, then we replace $D$ with $D'$.  We iterate this process until we arrive at a divisor $D$ such that the lemma cannot be applied for any pair of indices $i,j$.  Note that if $C$ is any $(-1)$-curve on $X$, then $C \cdot (-E_i+E_j) \geq -1$.  This implies that $D$ must be orthogonal to some $(-1)$-curve.  If $k\geq 3$, then applying an element of the Weyl group we may assume $D$ is orthogonal to $E_k$.  In this case, we reduce to studying a divisor class on the blowup at $k-1$ points. 

Likewise, if $k=2$ then by inspection $D$ is orthogonal to either $E_1$ or $E_2$.  That is, up to the Weyl group action, $D$ takes the form $D = dL - aE_1$ for some $0\leq a \leq d$.  In the end, we have reduced the proof of Proposition \ref{prop-delPezzoContained} to the following statement.

\begin{proposition}\label{prop-delPezzoRounding}
Let $D = dL - aE_1$ with $0\leq a \leq d$ be a nef divisor on $X = \Bl_{p_1,p_2} \P^2$, and let $r \geq 1$.  Then $D\in \Lambda_{r,-K_X}$.
\end{proposition}

\begin{remark}
It is crucial to regard $D$ as a divisor on the blowup at $2$ points instead of as a divisor on the blowup at $1$ point.  For example, consider $D = 2L$ and $r = 3$.  Then $D \cdot (-K_X) = 6$, so we need to express $2L$ as a sum of $3$ line bundles of $(-K_X)$-degree $2$ with no higher cohomology.  The only such line bundle on the blowup at $1$ point is $L-E_1$, so there is no way to achieve this.  On the other hand, on the blowup at $2$ points, there is the additional line bundle $E_1+E_2$ of $(-K_X)$-degree $2$, and  $$2L = (L-E_1)+(L-E_2)+(E_1+E_2).$$ 
\end{remark}

\begin{proof}[Proof of Proposition \ref{prop-delPezzoRounding}]
The proof is by induction on $r$.  The result is clear for $r=1$.  Let $r\geq 2$.  If $r=2$ and $D = L-E_1$, then $D = (L-2E_1)+E_1$, so we may ignore this case.  We will prove the following claim.

\emph{Claim:} Suppose it is not the case that $r=2$ and $D = L-E_1$.  Write $$  \frac{D \cdot (-K_X)}{r} = \frac{3d-a}{r} = m + \frac{p}{r} \qquad (0\leq p<r).$$ (We clearly have $m\geq 0$.) Then there is a line bundle $M$ with $M \cdot (-K_X)=m$ such that $M$ has no higher cohomology and $D-M$ is nef.

First observe that the claim implies the result.  Indeed, by induction and Lemma \ref{lem-upshift} it follows that $D-M \in \Lambda_{r-1,-K_X}$.  We have $$\frac{(D-M) \cdot (-K_X)}{r-1} = m+\frac{p}{r-1},$$ so $D-M$ can be written as $c_1(\cE)$ for a good sum $\cE$ of $p$ line bundles of $(-K_X)$-degree $m+1$ and $(r-1)-p$ line bundles of $(-K_X)$-degree $m$.  Then $D = c_1(\cE \oplus M)$, and $\cE\oplus M$ is a good sum.

Next we prove the claim by constructing the line bundle $M$.  Note that the claim is trivial if $m=0$, since then we can take $M=0$.  In what follows we assume $m>0$.  Write $$m = 3 s+t \qquad (0\leq t <  3).$$ For $0\leq t < 3$ we define line bundles $M_t$\ as follows:
\begin{align*}
M_0 &= 0\\
M_1 &= E_1\\
M_2 &= E_1+E_2
\end{align*}
noting that $M_t .(-K_X) = t$.  Then we put $$B = sL+M_t,$$ and observe $B.(-K_X) = m$.  The line bundle $B$ is the \emph{basic line bundle of $(-K_X)$-degree $m$}.  Note that $B$ has no higher cohomology.

We now modify $B$ by adding some number $\alpha$ of copies of $L-3E_1$ and some number $\beta$ of copies of $L-2E_1-E_2$ to it; call such a line bundle $$B' = B + \alpha(L-3E_1)+\beta (L-2E_1-E_2) =  d'L - a'E_1-b'E_2.$$  Then $B'$ still has $B'.(-K_X) = m$ since $L-3E_1$ and $L-2E_1-E_2$ are both orthogonal to $K_X$.  For $B'$ to be the desired line bundle $M$ that proves the claim, we need to choose $\alpha$ and $\beta$ so that $B'$ has no higher cohomology and $D-B'$ is nef.  

For $D-B'$ to be nef we must have $(D-B') \cdot E_2 \geq 0$, so we have $-1\leq b'\leq 0$.  Then either $\beta = 0$, or $t=2$ and $\beta = 1$. Similarly, considering $E_1$ shows $-1\leq a'\leq a$.  Now suppose we first make $\beta$ as large as possible, then make $\alpha$ as large as possible, without violating the inequalities $b' \leq 0$ or $a'\leq a$.  We claim that then $D-B'$ is nef, i.e. that $(D-B') \cdot (L-E_1-E_2) \geq 0$.  There are a couple cases to consider.

\emph{Case 1: $t=2$ and $a=0$.} In this case $\alpha=\beta =0$.  We have $D \cdot (L-E_1-E_2)=d$ and $B' = sL+E_1+E_2$, so $B' \cdot (L-E_1-E_2) = s+2$.  Then $$3(d-(s+2)) = 3d-m+t-6=mr+p-m-4 = m(r-1)+p-4.$$  Now $m\equiv 2\pmod 3$ and $m> 0$, so $d\geq s+2$ unless $m=2$, $s=0$, and $d=1$.  But then $m = \lfloor 3/r\rfloor$, so this is impossible since $r\geq 2$.  Therefore $d\geq s+2$ and $D-B'$ is nef.  Note that also $B'$ has no higher cohomology, so we may take $M=B'$ in this case to complete the proof.

\emph{Case 2: $t\neq 2$ or $a > 0$.}  If $t = 2$ but $a>0$ then we have $\beta =1$, and therefore $b'=0$.  We also have $b'=0$ if $t \neq 2$, so $b'=0$ in every case.  Then $B'$ takes one of the following three forms: $$
d'L - (a-2)E_1  \qquad d'L - (a-1)E_1 \qquad d'L - aE_1.
$$ In order to have $(D-B') \cdot (L-E_1-E_2)\geq 0$, we will need to compare $d'$ with $d$.    To do this we first compute $d'$.  Observe that the total number of line bundles added to the basic line bundle is $$\alpha+\beta = \left\lfloor \frac{a+t}{3}\right\rfloor;$$ write $$a+t = 3(\alpha+\beta) + \ell \qquad (0\leq \ell <3).$$ Then $$d' = s + \alpha + \beta,$$ and $$3(d-d') = 3d-3s-3(\alpha+\beta) = (mr+a+p)-(m-t)-(a+t-\ell)=m(r-1)+p+\ell.$$ Therefore $d'<d$ in any case, which implies $D-B'$ is nef unless $B' = d'L - (a-2)E_1$.  In this case we need the stronger inequality $d'\leq d-2$, or equivalently we must show $$m(r-1)+p+\ell >3.$$  But if $B' = d'L - (a-2)E_1$ then $\ell = 2$, and the only way the inequality can fail is if $m=1$, $r=2$, and $p=0$.  In this case $3d-a=2$, and since $a\leq d$ we have $d=a=1$.  Thus this case is the special case $D=L-E_1$, $r=2$, which we have already excluded.  Therefore $D-B'$ is nef.

Finally we are ready to construct the line bundle $M$ that proves the claim in case $t\neq 2$ or $a>0$.  Starting from the basic line bundle $B$, add a single copy of $L-2E_1-E_2$ if it won't violate the inequality $b'\leq 0$ (this won't violate the inequality $a'\leq a$ since $a>0$ if $t\neq 2$).   After that, repeatedly add copies of $L-3E_1$.  Sometime before the inequality $a'\leq a$ is violated, we will have $B' \cdot (L-E_1-E_2)\leq D \cdot (L-E_1-E_2)$. Let $M$ be the first $B'$ where this inequality holds.  Each copy of $L-2E_1-E_2$ or $M-3E_1$ that is added decreases the intersection number with $M-E_1-E_2$ by $2$.  Thus we will additionally have $M \cdot (L-E_1-E_2)\geq D \cdot (L-E_1-E_2)-1\geq -1$.  Since $b'=0$, this implies $M$ has no higher cohomology and $D-M$ is nef.
\end{proof}

\bibliographystyle{plain}

\end{document}